\documentclass[11pt,british]{paper}
\usepackage[T1]{fontenc}
\usepackage[latin9]{inputenc}
\usepackage[letterpaper]{geometry}
\usepackage{float}
\usepackage{amsthm}
\usepackage{amsmath}
\usepackage{amssymb}
\usepackage[authoryear]{natbib}

\makeatletter

\providecommand{\tabularnewline}{\\}
\floatstyle{ruled}
\newfloat{algorithm}{tbp}{loa}
\providecommand{\algorithmname}{Algorithm}
\floatname{algorithm}{\protect\algorithmname}

\numberwithin{equation}{section}
\numberwithin{figure}{section}
\theoremstyle{plain}
\newtheorem{thm}{\protect\theoremname}

\bibpunct{[}{]}{;}{a}{}{,}

\makeatother

\usepackage{babel}
\providecommand{\theoremname}{Theorem}

\begin{document}

\title{Optimising a nonlinear utility function in multi-objective integer
programming}

\author{Melih Ozlen, Meral Azizo\u{g}lu, Benjamin A.~Burton}
\maketitle
\begin{abstract}
In this paper we develop an algorithm to optimise a nonlinear utility
function of multiple objectives over the integer efficient set. Our
approach is based on identifying and updating bounds on the individual
objectives as well as the optimal utility value. This is done using
already known solutions, linear programming relaxations, utility function
inversion, and integer programming. We develop a general optimisation
algorithm for use with $k$ objectives, and we illustrate our approach
using a tri-objective integer programming problem.\end{abstract}

\section{Introduction}

The majority of studies reported in the optimisation literature consider
a single objective, such as minimising cost or maximising profit.
However, in practice, there are usually many objectives that need
to be considered simultaneously. In particular, the increasing effect
of globalisation brings safety, environmental impact and sustainability
issues, and hence their related performance measures, into consideration.
The practical aim is to find solutions that are not only profitable
but also safe, green and sustainable. 

Multi-Objective Integer Programming (MOIP) considers discrete feasible
sets defined by integer variables. The main focus of literature on
MOIP has been on enumerating the entire integer nondominated set,
or on optimising a single linear utility function. However, many practical
situations require the optimisation of a nonlinear utility function
that combines multiple objectives into one. Prominent applications
of such problems include, but are not limited to, pricing, routing,
production planning, resource allocation, portfolio selection, capital
budgeting, and designing networks. In such applications a utility
function can capture precisely the decision maker's preferences for
how to balance conflicting objectives, such as cost versus environmental
impact in routing problems, or profit versus risk in portfolio selection.

Besides their practical importance, these optimisation problems are
theoretically challenging as they---even their simpler single-objective
versions---fall into the class of NP-hard problems. Despite their
practical and theoretical importance, there is no reported research
on such problems, at least to the best of our knowledge. Recognising
this gap in the literature, in this paper we address the optimisation
of an explicitly defined nonlinear utility function of multiple objectives
over the integer efficient set, and provide a general framework for
its resolution. 

Throughout this paper we assume that the utility function is strictly
increasing with respect to each individual objective. Such monotonicity
is standard for a utility function, whose role is to combine many
individual objectives into a single global ``utility objective''.

A MOIP problem defines a discrete feasible set of efficient points
and corresponding nondominated objective vectors, and the optimal
value of our nonlinear utility function can always be found in this
nondominated set. A na{\"i}ve solution to our optimisation problem
could therefore be to generate all nondominated objective vectors
for the MOIP problem, and then to evaluate the utility function in
each case. However, this na{\"i}ve approach would be highly impractical,
since the number of nondominated vectors can be extremely large in
general.

Recognising this, we develop a more sophisticated approach to optimise
a nonlinear utility function over the integer efficient set, in which
we generate only a smaller subset of nondominated objective vectors.
To avoid the generation of unpromising solutions, we compute objective
optimality bounds by combining the best solution found so far with
the nonlinear utility function. To generate the subset of promising
nondominated objective vectors, we use an algorithm by \citet{ozlenazizoglu2009ejor}
that recursively solves MOIP problems with fewer objectives. 

The rest of the paper is organised as follows. Section \ref{sec:lit}
reviews the related literature. In Section \ref{sec:tri} we explain
our algorithm in full and prove its correctness. Section \ref{sec:example}
offers a detailed illustration of the workings of the algorithm, using
an instance of a tri-objective integer programming problem. We conclude
in Section \ref{sec:conclusion}.

\section{Literature\label{sec:lit}}

The best-studied cases of MOIP are Multi-Objective Combinatorial Optimisation
(MOCO) problems. These are special cases of MOIP that have special
constraint set structures. \citet{ehrgottgadbileux2000orspec,ehrgottgandibleux2004top}
provide rich surveys of MOCO studies that use exact and approximate
approaches, respectively. They address some special problem structures
and discuss their solution methodologies. \citet{ehrgottgadibleuxedsmcopt}
survey other MOCO problems including, but not limited to, nonlinear
programming, scheduling and multi-level programming. These recent
studies show a considerable growth in the MOCO literature. However,
research on generating all nondominated objective vectors for MOIP,
or on optimisation over that set, is still scarce.

\citet{kleinhannan1982ejor} develop an approach based on the sequential
solutions of the single-objective Integer Programming (IP) problems.
Their algorithm generates a subset, but not necessarily the whole
set, of all nondominated objective vectors. \citet{sylvacrema2004ejor}
improve the approach of \citet{kleinhannan1982ejor} by defining a
weighted combination of all objectives, and their approach guarantees
to generate all nondominated objective vectors. 

\citet{klamrothetaljogo} and \citet{ehrgott2006aor} study the general
MOIP problem. \citet{klamrothetaljogo} discuss the importance of
using upper bounds on the objective function values when generating
the nondominated set, and define composite functions to obtain such
bounds. To form the composite functions, they propose classical optimisation
methods such as cutting plane and branch and bound. \citet{ehrgott2006aor}
discusses various scalarisation techniques, and proposes a generalised
technique that encompasses the others as special cases. He also proposes
an elastic constraint method to identify all nondominated points,
whose power is dependent on the existence of sophisticated problem-dependent
methods for solving the single objective version.

\citet{ozlenazizoglu2009ejor} develop a general approach to generate
all nondominated objective vectors for the MOIP problem, by recursively
identifying upper bounds on individual objectives using problems with
fewer objectives. \citet{dhaenenslemesretalbi2010} present a similar
but parallel algorithm for MOCO problems, that again involves recursively
solving problems with fewer objectives. \citet{przybylskigandibleuxehrgotti2010joc}
and \citet{ozpeynirci20102302} propose similar algorithms to identify
all extreme nondominated objective vectors for a MOIP problem. They
both utilise a weighted single-objective function and partition the
weight space to identify the set of extreme supported nondominated
objective vectors, adapting a method first proposed for Multi-Objective
Linear Programming (MOLP) by \citet{bensonSun2000,bensonsun2002}.
\citet{przybylskigandibleuxehrgott2010do} extend the two phase method,
originally developed to generate the nondominated set for bi-objective
problems, to handle problems with three or more objectives. 

There are some recent studies proposing preference-based methods for
specific MOCO problems. \citet{lehuedeetal2006} optimise a utility
function in Multi-Attribute Utility Theory (MAUT) that is composed
of piecewise-linear utility functions and a Choquet integral in a
constraint programming-based algorithm. \citet{PernySpanjaard2005}
propose algorithms that consider any preorder for a multiobjective
minimum weight spanning tree problem. For the same problem, \citet{Galandetal2010ejor}
propose a branch and bound algorithm minimising a utility function
composed of a concave Choquet integral and a partial convex utility
function. \citet{galandetal2010lnems} also propose a branch and bound
algorithm maximising a convex Choquet integral function for the multiobjective
knapsack problem. Contrary to these preference-based studies, the
method we propose in this paper is not problem-specific: it can be
used with any MOIP problem, and requires no particular assumption
about the utility function beyond the natural assumption of monotonicity.

There are few studies dealing with general MOIP problems in which
the aim is to optimise a function. \citet{abbaschaabane2006} and
\citet{jorge2009} deal with optimising a linear function over the
efficient set of a MOIP problem.

\section{The algorithm\label{sec:tri}}

In its general form, our algorithm optimises a nonlinear utility function
of $k$ objectives over the integer programming efficient set. This
problem can be defined precisely as:\smallskip{}

Min~$G(f_{1}(x),f_{2}(x),\ldots,f_{k}(x))$

s.t.~$x\in X$,

where $X$ is the set of feasible points defined by $Ax=b$, $x_{j}\geq0$
and $x_{j}\mathbb{\in Z}$ for all $j\in\{1,2,\ldots,n\}$.\smallskip{}

The individual objectives are defined as $f_{1}(x)=\sum_{j=1}^{n}c_{1j}x_{j}$,
$f_{2}(x)=\sum_{j=1}^{n}c_{2j}x_{j}$, \ldots, and $f_{k}(x)=\sum_{j=1}^{n}c_{kj}x_{j}$,
where $c_{ij}\in\mathbb{Z}$ for all $i\in\{1,2,\ldots,k\}$ and $j\in\{1,2,\ldots,n$\}.
The nonlinear utility function $G(f_{1}(x),f_{2}(x),\ldots,f_{k}(x))$
is assumed to be continuous and strictly increasing in each individual
objective function $f_{1}(x)$, $f_{2}(x)$, \ldots, $f_{k}(x)$.

We refer to $G$ as a \emph{utility function} because it combines
the multiple objectives $f_{1},\ldots,f_{k}$. However, we \emph{minimise}
$G$ for consistency with other authors such as \citet{kleinhannan1982ejor}
and \citet{przybylskigandibleuxehrgott2010do}.

A point $x'\in X$ is called \emph{$k$-objective efficient} if and
only if there is no $x\in X$ such that $f_{i}(x)\leq f_{i}(x')$
for each $i\in\{1,2,\ldots,k\}$ and $f_{i}(x)<f_{i}(x')$ for at
least one $i$. The resulting objective vector $(f_{1}(x'),f_{2}(x'),\ldots,f_{k}(x'))$
is said to be \emph{$k$-objective nondominated}. One typically works
in objective space instead of variable space, since each nondominated
objective vector might correspond to a large number of efficient points
in variable space.

A point $x'\in X$ is called\emph{ optimal} if and only if there is
no $x\in X$ for which $G(f_{1}(x),f_{2}(x),\allowbreak\ldots,\allowbreak f_{k}(x))<G(f_{1}(x'),f_{2}(x'),\ldots,f_{k}(x'))$.
Because our utility function is strictly increasing, any optimal point
must also be $k$-objective efficient in variable space, and must
yield a $k$-objective nondominated vector in objective space.

Our proposed method of finding an optimal point is based on a shrinking
set of bounds for the $k$ individual objectives. We update the lower
bounds using linear programming relaxations. Where possible, we update
the upper bounds by using the lower bounds and inverting the utility
function; where necessary, we update the upper bounds using the algorithm
of \citet{ozlenazizoglu2009ejor}. The shrinking bounds allow us to
avoid unpromising portions of the integer efficient set, and the method
of updating these bounds is designed to solve integer programs only
when absolutely necessary. 

Algorithm \ref{Flo:alg} gives the stepwise description of our procedure
to find an optimal point for the nonlinear utility function $G(f_{1}(x),f_{2}(x),\ldots,f_{k}(x))$.
Key variables that we use in this algorithm include:
\begin{itemize}
\item $G^{BEST}$, the best known value of the utility function;
\item $f_{i}^{LB}$ and $f_{i}^{UB}$, the current lower and upper bounds
on the values of the individual objective functions.
\end{itemize}
{\small }
\begin{algorithm}
{\small \caption{Optimising $G$ over the efficient set of a MOIP problem}
\label{Flo:alg}}{\small \par}

{\small Step 0. Find some initial set of points $I$.}{\small \par}

{\small \hspace{1.25cm}Initialise $G^{BEST}=\min_{x\in I}G(f_{1}(x),f_{2}(x),\ldots,f_{k}(x))$.}{\small \par}

{\small \smallskip{}
}{\small \par}

{\small \hspace{1.25cm}Solve Min $f_{i}(x)$ s.t $x\in X$ for each
$i\in\{1,2,\ldots,k\}$.}{\small \par}

{\small \hspace{1.25cm}Set each $f_{i}^{LB}$ to the corresponding
optimal objective value.}{\small \par}

{\small \hspace{1.25cm}Set each $f_{i}^{UB}$ to $\infty$.}{\small \par}

{\small \medskip{}
}{\small \par}

{\small Step 1. If $G(f_{1}^{LB},f_{2}^{LB},\ldots,f_{k}^{LB})\geq G^{BEST}$
then STOP.}{\small \par}

{\small \hspace{1.25cm}For each objective $i\in\{1,2,\ldots,k\}$,
find $f_{i}^{A}$ that solves}{\small \par}

{\small \hspace{1.25cm}\quad{}$G(f_{1}^{LB},\ldots,f_{i-1}^{LB},f_{i}^{A},f_{i+1}^{LB},\ldots,f_{k}^{LB})=G^{BEST}$.}{\small \par}

{\small \hspace{1.25cm}\quad{}If this is impossible because $G(f_{1}^{LB},\ldots,f_{i-1}^{LB},z,f_{i+1}^{LB},\ldots,f_{k}^{LB})<G^{BEST}$
for all $z$,}{\small \par}

{\small \hspace{1.25cm}\quad{}\quad{}set $f_{i}^{A}=\infty$.}{\small \par}

{\small \hspace{1.25cm}Set $f_{i}^{UB}=\min(\left\lfloor f_{i}^{A}\right\rfloor ,f_{i}^{UB})$
for each $i\in\{1,2,\ldots,k\}$.}{\small \par}

{\small \medskip{}
}{\small \par}

{\small Step 2. For each objective $i\in\{1,2,\ldots,k\}$:\smallskip{}
}{\small \par}

{\small \hspace{1.25cm}\quad{}Solve the LP relaxation of}{\small \par}

{\small \hspace{1.25cm}\quad{}\quad{}Min $f_{i}(x)$ s.t. $x\in X$,
$f_{1}(x)\leq f_{1}^{UB}$, $f_{2}(x)\leq f_{2}^{UB}$, \ldots, and
$f_{k}(x)\leq f_{k}^{UB}$.}{\small \par}

{\small \hspace{1.25cm}\quad{}Let $f_{i}^{LP}=f_{i}(x^{*})$ be
the optimal objective value and $x^{\ast}$ be the corresponding}{\small \par}

{\small \hspace{1.25cm}\quad{}\quad{}optimal solution, and set
$f_{i}^{LB}=\left\lceil f_{i}^{LP}\right\rceil $.}{\small \par}

{\small \hspace{1.25cm}\quad{}If $x^{*}$ is integer and $G(f_{1}(x^{*}),f_{2}(x^{*}),\ldots,f_{k}(x^{*}))<G^{BEST}$,}{\small \par}

{\small \hspace{1.25cm}\quad{}\quad{}set $G^{BEST}=G(f_{1}(x^{*}),f_{2}(x^{*}),\ldots,f_{k}(x^{*}))$.}{\small \par}

{\small \smallskip{}
}{\small \par}

{\small \hspace{1.25cm}If the lower bound $f_{i}^{LB}$ is updated
for any objective $i$ then go to Step 1.}{\small \par}

{\small \medskip{}
}{\small \par}

{\small Step 3. Use \citet{ozlenazizoglu2009ejor} to update the upper
bound $f_{k}^{UB}$. Specifically:\smallskip{}
}{\small \par}

{\small \hspace{1.25cm}Begin generating all nondominated objective
vectors for the $(k-1)$-objective MOIP problem}{\small \par}

{\small \hspace{1.25cm}\quad{}Min $f_{1}(x)$, $f_{2}(x)$, \ldots,
$f_{k-1}(x)$}{\small \par}

{\small \hspace{1.25cm}\quad{}s.t. $x\in X$, $f_{1}(x)\leq f_{1}^{\mathit{{UB}}}$,
$f_{2}(x)\leq f_{2}^{\mathit{{UB}}}$, \ldots, $f_{k}(x)\leq f_{k}^{\mathit{{UB}}}$.}{\small \par}

{\small \hspace{1.25cm}Each time we generate a $(k-1)$-objective
vector, fix the first $k-1$ objectives to the}{\small \par}

{\small \hspace{1.25cm}\quad{}values found and minimise $f_{k}(x)$.
This two-step ``lexicographic optimisation'' yields a}{\small \par}

{\small \hspace{1.25cm}\quad{}$k$-objective nondominated objective
vector for the original problem.}{\small \par}

{\small \hspace{1.25cm}If no feasible point exists then STOP.\smallskip{}
}{\small \par}

{\small \hspace{1.25cm}Each time we generate a nondominated $k$-objective
vector $f^{\ast}$,}{\small \par}

{\small \hspace{1.25cm}\quad{}test whether $G(f_{1}^{\ast},f_{2}^{\ast},\ldots,f_{k}^{\ast})<G^{BEST}$.}{\small \par}

{\small \hspace{1.25cm}\quad{}If true, set $G^{BEST}=G(f_{1}^{\ast},f_{2}^{\ast},\ldots,f_{k}^{\ast})$
and go to Step 1 immediately.\smallskip{}
}{\small \par}

{\small \hspace{1.25cm}Let $S$ be the set of all nondominated objective
vectors that were generated above.}{\small \par}

{\small \hspace{1.25cm}Set $f_{k}^{UB}=\max_{f\in S}f_{k}-1$ and
go to Step 2.}
\end{algorithm}
{\small \par}

Algorithm \ref{Flo:alg} terminates with a $G^{BEST}$ value that
is the minimum of $G(f_{1}(x),f_{2}(x),\ldots,f_{k}(x))$ among all
efficient points; moreover, each time we update $G^{\mathit{{BEST}}}$
we can record the corresponding efficient $x\in X$. Stated formally:
\begin{thm}
\label{thm:tri}Algorithm \ref{Flo:alg} finds the minimum value of
$G(f_{1}(x),f_{2}(x),\ldots,f_{k}(x))$ among all integer efficient
points for the MOIP problem, and also identifies a corresponding $x\in X$
that attains this minimum.\end{thm}
\begin{proof}
As the algorithm runs we maintain the following invariants:
\begin{itemize}
\item The utility $G^{BEST}$ is obtainable; that is, $G^{BEST}=G(f_{1}(x),f_{2}(x),\ldots,f_{k}(x))$
for some feasible integer point $x\in X$. 
\item Either $G^{BEST}$ is already equal to the optimal utility, or else
the optimal utility can be achieved for some point $x\in X$ with
$f_{i}^{LB}\leq f_{i}(x)\leq f_{i}^{UB}$ for each objective $i\in\{1,\ldots,k\}$.
\end{itemize}
In essence, the lower and upper bounds $f_{i}^{LB}$ and $f_{i}^{UB}$
are used to bound the region of the integer efficient set that remains
to be examined. It is easy to see that these invariants hold:
\begin{itemize}
\item In Step 0, each $f_{i}^{UB}$ is $\infty$, each $f_{i}^{LB}$ is
the global minimum for $f_{i}(x)$, and $G^{BEST}$ is obtained from
a known point $x\in X$.
\item In Step 1, any point $x\in X$ with $f_{i}(x)>f_{i}^{A}$ must have
a utility worse than $G^{BEST}$. This is because $G$ is strictly
increasing, and so any such $x$ must satisfy $G(f_{1}(x),\ldots,f_{i}(x),\ldots,\allowbreak f_{k}(x))>G(f_{1}^{\mathit{{LB}}}(x),\ldots,f_{i}^{A},\ldots,f_{k}^{\mathit{{LB}}}(x))$.
The revised upper bound of $\lfloor f_{i}^{A}\rfloor$ is valid because
each $f_{i}(x)\in\mathbb{Z}$.
\item In Step 2, the revised lower bounds are valid because each optimal
LP value $f_{i}^{LP}$ is equal to or better than the corresponding
optimal IP value. Again we can round $\left\lceil f_{i}^{\mathit{{LP}}}\right\rceil $
because each $f_{i}(x)\in\mathbb{Z}$.
\item In Step 3, any revision to $G^{BEST}$ is valid because it comes from
an efficient point: in particular, \citet{ozlenazizoglu2009ejor}
show that each solution $x\in X$ to the lexicographic optimisation
in Step 3 is a $k$-objective efficient point. The revision to $f_{k}^{UB}$
is valid because \citet{ozlenazizoglu2009ejor} show there cannot
exist any other efficient point having objective function value $f_{k}(x)$
between $\max_{f\in S}f_{k}$ and the previous value of $f_{k}^{UB}$. 
\end{itemize}
Note that $f_{i}^{A}$ always exists in Step~1 because $G$ is continuous,
and because our invariants give $G(f_{1}^{LB},\ldots,f_{i-1}^{LB},f_{i}^{LB},f_{i+1}^{LB},\ldots,f_{k}^{LB})\leq G^{BEST}$;
moreover, this $f_{i}^{A}$ is simple to find using standard univariate
search techniques.

To prove that the algorithm terminates: Even if no bounds are updated
in Steps 1 or 2, the procedure of \citet{ozlenazizoglu2009ejor} will
reduce the bound $f_{k}^{UB}$ in Step 3. This ensures that the bounds
shrink during every loop through the algorithm, and because these
shrinking bounds are integers we must terminate after finitely many
steps.

To prove that the algorithm gives the optimal utility: Upon termination,
either $G^{BEST}$ is at least as good as anything obtainable within
our bounds $f_{i}^{LB}\leq f_{i}(x)\leq f_{i}^{UB}$ (if we STOP in
Step 1 or 2), or else these bounds have been reduced so far that the
remaining integer efficient set is empty (if we STOP in Step 3). Either
way, we know from our invariants that $G^{BEST}$ is obtainable and
that no better utility is possible.
\end{proof}
Some final notes:
\begin{itemize}
\item In Step~2, hoping for an integer solution $x^{\ast}$ is optimistic.
However, this simple test is cheap, and it can speed up the computation
in the case where $x^{\ast}\in\mathbb{Z}$.
\item \citet{ozlenazizoglu2009ejor} implement Step~3 using $\epsilon$
multipliers for objectives, thereby transforming the lexicographic
optimisation from Step~3 into a set of single-objective problems;
see the paper for details. Other algorithms for generating nondominated
objective vectors and/or efficient points may be used here instead
of \citet{ozlenazizoglu2009ejor} (for instance, in cases where highly
efficient problem-specific algorithms are known).
\item Although the MOIP problem in Step~3 is the computational bottleneck
of the algorithm, it has only $k-1$ objectives, and we do not require
all solutions (since we exit step~3 as soon as the bound $G^{\mathit{{BEST}}}$
is updated). Both of these features make it significantly easier than
the initial $k$-objective problem.
\end{itemize}

\section{An example problem\label{sec:example}}

In this section we illustrate our approach on a concrete example with
$k=3$. This is a tri-objective assignment problem of size $5\times5$,
with the nonlinear utility function $G(f_{1}(x),f_{2}(x),f_{3}(x))=f_{1}(x){}^{3}+f_{2}(x){}^{3}+f_{3}(x){}^{3}$.
The individual objective coefficients for the problem, taken from
\citet{ozlenazizoglu2009ejor}, are provided in Table \ref{Flo:example-data}.

\begin{table}[h]
\caption{Objective coefficients for the example problem}

\begin{tabular}{|c|c|c|c|c|c|c|c|c|c|c|c|c|c|c|c|c|c|c|c|}
\cline{1-6} \cline{8-13} \cline{15-20} 
$c_{1}$ & $1$ & $2$ & $3$ & $4$ & $5$ &  & $c_{2}$ & $1$ & $2$ & $3$ & $4$ & $5$ &  & $c_{3}$ & $1$ & $2$ & $3$ & $4$ & $5$\tabularnewline
\cline{1-6} \cline{8-13} \cline{15-20} 
$1$ & 99 & 19 & 74 & 55 & 41 &  & $1$ & 28 & 39 & 19 & 42 & 7 &  & $1$ & 29 & 67 & 2 & 90 & 7\tabularnewline
\cline{1-6} \cline{8-13} \cline{15-20} 
$2$ & 23 & 81 & 93 & 39 & 49 &  & $2$ & 66 & 98 & 49 & 83 & 42 &  & $2$ & 84 & 37 & 64 & 64 & 87\tabularnewline
\cline{1-6} \cline{8-13} \cline{15-20} 
$3$ & 66 & 21 & 63 & 24 & 38 &  & $3$ & 73 & 26 & 42 & 13 & 54 &  & $3$ & 54 & 11 & 100 & 83 & 61\tabularnewline
\cline{1-6} \cline{8-13} \cline{15-20} 
$4$ & 65 & 41 & 7 & 39 & 66 &  & $4$ & 46 & 42 & 28 & 27 & 99 &  & $4$ & 75 & 63 & 69 & 96 & 3\tabularnewline
\cline{1-6} \cline{8-13} \cline{15-20} 
$5$ & 93 & 30 & 5 & 4 & 12 &  & $5$ & 80 & 17 & 99 & 59 & 68 &  & $5$ & 66 & 99 & 34 & 33 & 21\tabularnewline
\cline{1-6} \cline{8-13} \cline{15-20} 
\end{tabular}

\label{Flo:example-data}
\end{table}

The iterations of Algorithm \ref{Flo:alg} are summarised in Table
\ref{Flo:example-img}. Columns in this table show the solutions of
LP and IP problems in objective space, and the updated values of lower
and upper bounds on the individual objective functions and the utility
function as they appear in Algorithm \ref{Flo:alg}. For the columns
representing bounds, an empty cell indicates that the value has not
changed from the line above.

\begin{table}
\tiny\renewcommand{\arraystretch}{1.5}%
\begin{tabular}{|c|c|l|c|c|c|c|c|c|c|c|c|c|c|}
\hline 
Step & \#IP &  & $f_{1}(x)$ & $f_{2}(x)$ & $f_{3}(x)$ & $G^{BEST}$ & $f_{1}^{LB}$ & $f_{2}^{LB}$ & $\mbox{\ensuremath{f_{3}^{LB}}}$ & $G^{LB}$ & $f_{1}^{UB}$ & $f_{2}^{UB}$ & $f_{3}^{UB}$\tabularnewline
\hline 
$0$ & $1$ & IP Min $f_{1}(x)$ & $86$ & $214$ & $324$ & $44,448,624$ & $86$ & $-\infty$ & $-\infty$ & $636,056$ & $\infty$ & $\infty$ & $\infty$\tabularnewline
\hline 
$0$ & $2$ & IP Min $f_{2}(x)$ & $209$ & $128$ & $367$ &  &  & $128$ &  & $2,733,208$ &  &  & \tabularnewline
\hline 
$0$ & $3$ & IP Min $f_{3}(x)$ & $291$ & $348$ & $129$ &  &  &  & $129$ & $4,879,897$ &  &  & \tabularnewline
\hline 
$1$ &  & Find $f_{1}^{A},f_{2}^{A},f_{3}^{A}$ &  &  &  &  &  &  &  &  & $342$ & $346$ & $346$\tabularnewline
\hline 
$2$ &  & LP Min $f_{i}(x),i=1,2,3$ & $86$ & $130.2$ & $129.1$ &  &  & $131$ & $130$ & $5,081,147$ &  &  & \tabularnewline
\hline 
$1$ &  & Find $f_{1}^{A},f_{2}^{A},f_{3}^{A}$ &  &  &  &  &  &  &  &  &  &  & \tabularnewline
\hline 
$3$ & $4$ & IP Min $f_{1}(x),f_{2}(x)$  & $86$ & $214$ & $324$ &  &  &  &  &  &  &  & \tabularnewline
\hline 
$3$ & $5$ & IP Min $f_{1}(x),f_{2}(x)$ & $96$ & $186$ & $204$ & $15,809,256$ &  &  &  &  &  &  & \tabularnewline
\hline 
$1$ &  & Find $f_{1}^{A},f_{2}^{A},f_{3}^{A}$ &  &  &  &  &  &  &  &  & $224$ & $234$ & $234$\tabularnewline
\hline 
$2$ &  & LP Min $f_{i}(x),i=1,2,3$ & $93.5$ & $169.8$ & $157.3$ &  & $94$ & $170$ & $158$ & $9,687,896$ &  &  & \tabularnewline
\hline 
$1$ &  & Find $f_{1}^{A},f_{2}^{A},f_{3}^{A}$ &  &  &  &  &  &  &  &  & $190$ & $222$ & $215$\tabularnewline
\hline 
$2$ &  & LP Min $f_{i}(x),i=1,2,3$ & $95.1$ & $178.4$ & $167.6$ &  & $96$ & $179$ & $168$ & $11,361,707$ &  &  & \tabularnewline
\hline 
$1$ &  & Find $f_{1}^{A},f_{2}^{A},f_{3}^{A}$ &  &  &  &  &  &  &  &  & $174$ & $216$ & $209$\tabularnewline
\hline 
$2$ &  & LP Min $f_{i}(x),i=1,2,3$ & $95.6$ & $181.3$ & $173.7$ &  &  & $182$ & $174$ & $12,181,328$ &  &  & \tabularnewline
\hline 
$1$ &  & Find $f_{1}^{A},f_{2}^{A},f_{3}^{A}$ &  &  &  &  &  &  &  &  & $165$ & $212$ & $207$\tabularnewline
\hline 
$2$ &  & LP Min $f_{i}(x),i=1,2,3$ & $95.8$ & $182.4$ & $177.6$ &  &  & $183$ & $178$ & $12,652,975$ &  &  & \tabularnewline
\hline 
$1$ &  & Find $f_{1}^{A},f_{2}^{A},f_{3}^{A}$ &  &  &  &  &  &  &  &  & $159$ & $210$ & $206$\tabularnewline
\hline 
$2$ &  & LP Min $f_{i}(x),i=1,2,3$ & $95.8$ & $183.1$ & $179.7$ &  &  & $184$ & $180$ & $12,946,240$ &  &  & \tabularnewline
\hline 
$1$ &  & Find $f_{1}^{A},f_{2}^{A},f_{3}^{A}$ &  &  &  &  &  &  &  &  & $155$ & $208$ & \tabularnewline
\hline 
$2$ &  & LP Min $f_{i}(x),i=1,2,3$ & $95.9$ & $183.5$ & $181.6$ &  &  &  & $182$ & $13,142,808$ &  &  & \tabularnewline
\hline 
$1$ &  & Find $f_{1}^{A},f_{2}^{A}$ &  &  &  &  &  &  &  &  & $152$ & $207$ & \tabularnewline
\hline 
$2$ &  & LP Min $f_{i}(x),i=1,2,3$ & $95.9$ & $183.7$ & $182.6$ &  &  &  & $183$ & $13,242,727$ &  &  & \tabularnewline
\hline 
$1$ &  & Find $f_{1}^{A},f_{2}^{A}$ &  &  &  &  &  &  &  &  & $151$ & $206$ & \tabularnewline
\hline 
$2$ &  & LP Min $f_{i}(x),i=1,2,3$ & $95.9$ & $183.7$ & $183.5$ &  &  &  & $184$ & $13,343,744$ &  &  & \tabularnewline
\hline 
$1$ &  & Find $f_{1}^{A},f_{2}^{A}$ &  &  &  &  &  &  &  &  & $149$ & $205$ & \tabularnewline
\hline 
$2$ &  & LP Min $f_{i}(x),i=1,2,3$ & $95.9$ & $183.7$ & $184.4$ &  &  &  & $185$ & $13,445,865$ &  &  & \tabularnewline
\hline 
$1$ &  & Find $f_{1}^{A},f_{2}^{A}$ &  &  &  &  &  &  &  &  & $148$ & $204$ & \tabularnewline
\hline 
$2$ &  & LP Min $f_{i}(x),i=1,2,3$ & $95.9$ & $183.8$ & $185.3$ &  &  &  & $186$ & $13,549,096$ &  &  & \tabularnewline
\hline 
$1$ &  & Find $f_{1}^{A},f_{2}^{A}$ &  &  &  &  &  &  &  &  & $146$ &  & \tabularnewline
\hline 
$2$ &  & LP Min $f_{i}(x),i=1,2,3$ &  & $183.9$ & $185.5$ &  &  &  &  &  &  &  & \tabularnewline
\hline 
$3$ & $6$ & IP Min $f_{1}(x),f_{2}(x)$  & $96$ & $186$ & $204$ &  &  &  &  &  &  &  & $203$\tabularnewline
\hline 
$3$ & $7$ & IP Min $f_{1}(x),f_{2}(x)$ & inf. &  &  &  &  &  &  &  &  &  & \tabularnewline
\hline 
$2$ &  & LP Min $f_{i}(x),i=1,2$ & $97.3$ & $184.6$ &  &  & $98$ & $185$ &  & $13,707,673$ &  &  & \tabularnewline
\hline 
$1$ &  & Find $f_{1}^{A},f_{2}^{A},f_{3}^{A}$ &  &  &  &  &  &  &  &  & $144$ & $203$ & \tabularnewline
\hline 
$2$ &  & LP Min $f_{i}(x),i=1,2,3$ & $97.3$ & $184.7$ & $186.4$ &  &  &  & $187$ & $13,812,020$ &  &  & \tabularnewline
\hline 
$1$ &  & Find $f_{1}^{A},f_{2}^{A}$ &  &  &  &  &  &  &  &  & $143$ & $202$ & \tabularnewline
\hline 
$2$ &  & LP Min $f_{i}(x),i=1,2,3$ & $97.3$ & $184.7$ & $187.3$ &  &  &  & $188$ & $13,917,489$ &  &  & \tabularnewline
\hline 
$1$ &  & Find $f_{1}^{A},f_{2}^{A}$ &  &  &  &  &  &  &  &  & $141$ & $201$ & \tabularnewline
\hline 
$2$ &  & LP Min $f_{i}(x),i=1,2,3$ & $97.3$ & $184.8$ & $188.3$ &  &  &  & $189$ & $14,024,086$ &  &  & \tabularnewline
\hline 
$1$ &  & Find $f_{1}^{A},f_{2}^{A}$ &  &  &  &  &  &  &  &  & $139$ & $200$ & \tabularnewline
\hline 
$2$ &  & LP Min $f_{i}(x),i=1,2,3$ & $97.3$ & $184.9$ & $189.2$ &  &  &  & $190$ & $14,141,817$ &  &  & \tabularnewline
\hline 
$1$ &  & Find $f_{1}^{A},f_{2}^{A}$ &  &  &  &  &  &  &  &  & $137$ &  & \tabularnewline
\hline 
$2$ &  & LP Min $f_{i}(x),i=2,3$ &  & $184.9$ & $189.4$ &  &  &  &  &  &  &  & \tabularnewline
\hline 
$3$ & $8$ & IP Min $f_{1}(x),f_{2}(x)$ & inf. &  &  & $15,809,256$ & $98$ & $185$ & $190$ & $14,131,817$ & $137$ & $200$ & $203$\tabularnewline
\hline 
\end{tabular}

\caption{Iteration details of Algorithm \ref{Flo:alg} on the example problem
instance}

\label{Flo:example-img}
\end{table}

For the initialisation in Step 0 the procedure solves three lexicographic
IPs, minimising objectives in the following lexicographic order: 1-2-3,
2-1-3, 3-1-2. There are of course many alternate methods of initialisation;
we use lexicographic IPs here because they produce a good spread of
nondominated objective vectors.

Step 1 then sets (and later updates) upper bounds on the individual
objective functions based on the current best solution, $G^{BEST}$.
Step 2 updates the lower bounds by solving the linear programming
relaxations with the upper bound constraints on the individual objective
function values. Steps 1 and 2 are iterated for as long as they continue
to update these lower and upper bounds. When the bounds cannot be
updated further, Step 3 generates tri-objective nondominated objective
vectors by generating a bi-objective nondominated set based on the
upper bounds $f_{1}^{UB}$, $f_{2}^{UB}$ and $f_{3}^{UB}$. If Step
3 is able to improve upon the best utility value $G^{BEST}$, it returns
to Step 1; otherwise it updates $f_{3}^{UB}$ and returns to Step
2. In the final iteration, where Step 3 fails to find any feasible
points within the current bounds, the entire algorithm terminates.

The optimal solution with $G(96,186,204)=15\,809\,256$ is identified
at an early stage but it takes a large number of iterations to prove
its optimality. We see from Table \ref{Flo:example-img} that our
shrinking bounds perform very well for this example: Algorithm \ref{Flo:alg}
requires the solution of just eight IPs to find the optimal utility
value. If we were to use the na\"ive method from the introduction
and generate all nondominated objective vectors then we would require
a total of 56 IPs to solve, as described in \citet{ozlenazizoglu2009ejor}.
This illustrates the way in which many IPs can be avoided (by eliminating
nondominated objective vectors without explicitly generating them)
using the shrinking bound techniques of Algorithm \ref{Flo:alg}.

\section{Conclusion\label{sec:conclusion}}

In this study we propose a general algorithm to optimise a nonlinear
utility function of multiple objectives over the integer efficient
set. As an alternative to the na\"ive method of generating and evaluating
all nondominated objective vectors, we restrict our search to a promising
subset of nondominated vectors by computing and updating bounds on
the individual objectives. The nondominated vectors within this promising
subset are generated using the algorithm of \citet{ozlenazizoglu2009ejor}.
As illustrated by the example in Section \ref{sec:example}, these
bounding techniques can significantly reduce the total number of IPs
to be solved. Because solving IPs is the most computationally expensive
part of the algorithm, we expect these bounding techniques to yield
a significant performance benefit for the algorithm as a whole.

For larger problems that remain too difficult to solve, Algorithm
\ref{Flo:alg} can be used as an \emph{approximation} algorithm. We
can terminate the algorithm at any time, whereupon $G^{BEST}$ and
$G(f_{1}^{LB},f_{2}^{LB},\ldots,f_{k}^{LB})$ will give upper and
lower bounds for the optimal utility, and we will have a feasible
point $x\in X$ for which $G(f_{1}(x),f_{2}(x),\ldots,f_{k}(x))=G^{BEST}$. 

We hope that this study stimulates future work in the field of multi-objective
optimisation. One promising direction for future research may be to
apply our algorithm to specific families of MOCO problems. The special
structure of the constraints in these families might help to improve
the efficiency of our algorithm for nonlinear utility functions.

\section*{Acknowledgements}

The third author is supported by the Australian Research Council under
the Discovery Projects funding scheme (project DP1094516). We thank
the anonymous reviewers for their suggestions.

\vspace{2cm}

\noindent Melih Ozlen\\
School of Mathematical and Geospatial Sciences, RMIT University\\
GPO Box 2476V, Melbourne VIC 3001, Australia\\
(melih.ozlen@rmit.edu.au)

\bigskip{}

\noindent Meral Azizo\u{g}lu\\
Department of Industrial Engineering, Middle East Technical University\\
Ankara 06531, Turkey\\
(meral@ie.metu.edu.tr)

\bigskip{}

\noindent Benjamin A.~Burton\\
School of Mathematics and Physics, The University of Queensland\\
Brisbane QLD 4072, Australia\\
(bab@maths.uq.edu.au)
\end{document}